\documentclass[11pt,a4paper]{amsart}
\usepackage[utf8]{inputenc}
\usepackage[T1]{fontenc}
\usepackage{amsmath,amssymb,amsthm}
\usepackage{enumitem}
\usepackage{xcolor}
\usepackage[colorlinks=true,linkcolor=blue!60!black,citecolor=blue!60!black,urlcolor=blue!60!black]{hyperref}
\usepackage[margin=2.6cm]{geometry}
\input xy
\xyoption{all}

\theoremstyle{plain}
\newtheorem{theorem}{Theorem}[section]
\newtheorem{proposition}[theorem]{Proposition}
\newtheorem{lemma}[theorem]{Lemma}
\newtheorem{corollary}[theorem]{Corollary}

\theoremstyle{definition}
\newtheorem{definition}[theorem]{Definition}
\newtheorem{remark}[theorem]{Remark}

\newcommand{\F}{\mathbb F}
\newcommand{\Z}{\mathbb Z}

\newcommand{\Zp}{\mathbb Z_p}
\newcommand{\PP}{\mathbb P}

\newcommand{\Spec}{\operatorname{Spec}}

\newcommand{\RH}{R\!H}

\title{$R$-trivial Galois cohomology classes\\ and cohomological $p$-dimension}
\author{Charles De Clercq}
\date{\today}

\begin{document}
\begin{abstract}
Given a prime $p$, a field $F$ of characteristic different from $p$ and a finite $p$-primary Galois module $V$, a Galois cohomology class in $H^n(F,V)$ is called $R$-trivial if it can be connected to $0$ by a class defined over an open subscheme of the projective line. All
classes of $H^n(F,V)$ are $R$-trivial if $F$ has cohomological $p$-dimension at most $n$, by the Bloch-Kato conjecture. We show that this bound is sharp: working over fields of iterated Laurent series, we produce, for every $n\geq 1$, a field $F$ of cohomological $p$-dimension
$n+1$ and a Galois module $V$ such that $H^n(F,V)$ contains classes which are not $R$-trivial. The module $V$ is a quotient of the permutation module on the affine hyperplanes of an elementary abelian $p$-group, arising as a quotient of the absolute Galois group of $F$. We also describe $R$-trivial classes as sums of corestrictions of cup products involving a Kummer class.
\end{abstract}
\maketitle

\section{Introduction}

Throughout, $p$ is a prime number and we denote by $\Gamma_F$ the absolute Galois group of a field $F$. By a finite $p$-module we mean a finite discrete $\Gamma_F$-module of $p$-power order. We denote by $H^n(F,-)$ Galois cohomology, identified with \'etale cohomology of $\Spec F$. For simplicity, we assume throughout that $F$ has characteristic $0$ and contains all roots of unity of $p$-power order; this is only to avoid Tate twists. We fix once and for all a generator of the Tate module $\Zp(1)=\varprojlim\mu_{p^m}$ and use it to identify every Tate twist $M(i)$ of a finite $p$-module $M$ with $M$ itself (in particular the Kummer class of $a\in F^*$ is viewed as an element $(a)\in H^1(F,\Z/p^m\Z)$).
 
\subsection{The question}
Let $V$ be a finite $p$-module and $n\ge1$. Following the terminology used for tori \cite{CTSan} and for torsors \cite{FR}, a class $\alpha\in H^n(F,V)$ is called \emph{$R$-trivial} if there is a dense open $U\subseteq\PP^1_F$ containing $0$ and $1$ and a class $\beta\in H^n(U,V)$ such that $\beta(0)=0$ and $\beta(1)=\alpha$ (Definition~\ref{requiv}). The $R$-trivial classes form a subgroup $\RH^n(F,V)$ of $H^n(F,V)$. Kummer theory shows that every class of $H^1(L,\F_p)$, $L/F$ finite, is $R$-trivial (the same holds for $H^1(F,\F_p[\Gamma_F/\Gamma_L])$ by Shapiro's lemma). In
degree $n$, Rost and Voevodsky theorem (Bloch-Kato conjecture), shows that all classes of $H^n(F,\F_p)$ are $R$-trivial. One can ask whether this extends to arbitrary coefficients:
 
\begin{quote}
\emph{Let $n\ge1$. Are all classes of $H^n(F,V)$ $R$-trivial, for every finite $p$-module $V$?}
\end{quote}

If $F$ is of cohomological $p$-dimension at most $n$, the answer is positive: $V$ is a quotient of a permutation module $P$, the map $H^n(F,P)\to H^n(F,V)$ is then surjective and $H^n(F,P)$ is generated by images of pure symbols by Shapiro's lemma and the Bloch-Kato conjecture. We show this condition is sharp by producing counterexamples to this question with fields of iterated Laurent series, of cohomological $p$-dimension $n+1$.

\begin{theorem}\label{main}
   Let $n\geq 1$ and $k$ be an algebraically closed field of characteristic $0$. Consider the field of iterated Laurent series $F=k((x_1))\dots ((x_{n+1}))$ and set $G=\mathrm{Gal}(F(x_1^{1/p},\dots,x_{n+1}^{1/p})/F)$.
   
View the quotient affine hyperplane module $\mathcal{Q}_G$ (Definition~\ref{quotdef}) as a $\Gamma_F$-module through the surjection $\Gamma_F\to G$. Then $H^n(F,\mathcal{Q}_G)$ contains a class which is not $R$-trivial.\end{theorem}

The precise statement is given by Theorem \ref{main2}.

\subsection{Affine hyperplane modules}

\begin{definition}\label{quotdef}
Let $E$ be a finite-dimensional vector space over $\F_p$ and let $\mathcal{A}_E$ be the set of its affine hyperplanes, on which $E$ acts by translation. The permutation module $\F_p[\mathcal{A}_E]$  fits into a short exact sequence 
\begin{equation}
0\longrightarrow \F_p \overset{\iota}{\longrightarrow} \F_p[\mathcal{A}_E]\longrightarrow \mathcal{Q}_E\longrightarrow 0\label{se}   
\end{equation}
 where $\iota$ sends $1$ to $\sum_{h\in \mathcal{A}_E}[h]$. The quotient $\mathcal{Q}_E$ is the quotient affine hyperplane module of $E$.  
\end{definition}

\begin{remark}
Denote by $\mathcal{H}_E$ the set of linear hyperplanes of $E$. Grouping affine hyperplanes according to their direction yields $\F_p[\mathcal{A}_E]=\bigoplus_{H\in \mathcal{H}_E}\F_p[E/H]$. If moreover $E=\mathrm{Gal}(L/F)$ is the Galois group of a Galois extension, Shapiro's lemma then identifies $H^i(F,\F_p[\mathcal{A}_E])$ with the direct sum  $\bigoplus_{H\in \mathcal{H}_E}H^i(L^H,\F_p)$ where $L^H/F$ is the cyclic extension of degree $p$ given by $H$.
\end{remark}

\subsection{Statement}

\begin{theorem}\label{main2}  Let $n\geq 1$ and $k$ be an algebraically closed field of characteristic $0$. Consider the field of iterated Laurent series $F=k((x_1))\dots ((x_{n+1}))$ and set $G=\mathrm{Gal}(F(x_1^{1/p},\dots,x_{n+1}^{1/p})/F)$.
View the quotient affine hyperplane module $\mathcal{Q}_G$ as a $\Gamma_F$-module through the surjection $\Gamma_F\to G$. Denoting by $\delta$ the boundary map associated to \eqref{se}, we have
    $$\RH^{n}(F,\mathcal{Q}_G)=\mathrm{ker}\Bigl(\delta:H^{n}(F,\mathcal{Q}_G)\longrightarrow H^{n+1}(F,\F_p)\Bigr),$$
which is a hyperplane in $H^{n}(F,\mathcal{Q}_G)\simeq \F_p^{(n+1)(|\mathcal{H}_G|-1)+1}$. In particular every class $\alpha$ of $H^{n}(F,\mathcal{Q}_G)$ such that $\delta(\alpha)=(x_1)\cup\cdots \cup (x_{n+1})$ is not $R$-trivial. 
\end{theorem}

\begin{remark}
The proofs can be extended to show that, for every $n\ge1$ and every $N>n$, there are non-$R$-trivial classes in degree $n$ over a field of cohomological $p$-dimension $N$ (namely $k((x_1))\cdots((x_{N}))$, with the module $\mathcal{Q}_G$ built from the first $n+1$ variables).
\end{remark}

\begin{remark}
The modules $\mathcal{Q}_E$ with $\dim E=2$ are the ones considered in \cite[\S13]{DCF2}.
Combined with Proposition~\ref{corsym} for $n=1$, the results there give, for an arbitrary
field $F$ as in our conventions, a stronger statement than the case $n=1$ of
Theorem~\ref{main2}: all classes of $H^1(F,V)$ are $R$-trivial for every finite $p$-module
$V$ if and only if $\mathrm{cd}_p(F)\le1$. For $n\ge2$ we do not know if the analogous
equivalence holds, between $R$-triviality of all classes of degree $n$ and
$\mathrm{cd}_p(F)\le n$.
\end{remark}

\section{$R$-trivial cohomology classes}

Recall that our base field $F$ has characteristic $0$. Let $X$ be an $F$-scheme, with structure morphism $\pi\colon X\to\Spec F$, and let $M$ be a finite $p$-module, viewed as an \'etale sheaf on $\Spec F$. We write $H^n(X,M)$ for the \'etale cohomology of the locally constant sheaf $\pi^*M$ on $X$. For a rational point $x\in X(F)$, viewed as a section
$x\colon\Spec F\to X$ of $\pi$, and a class $\gamma\in H^n(X,M)$, we denote by $$\gamma(x)=x^*\gamma\in H^n(F,M)$$
the value of $\gamma$ at $x$. We fix a coordinate $t$ on $\PP^1_F$ and denote by $0$ and $1$ the rational points $t=0$ and $t=1$.

\begin{definition}\label{requiv}
    Let $M$ be a finite $p$-module and $n\geq 1$. Two classes $\alpha$ and $\beta$ in $H^n(F,M)$ are $R$-equivalent if there is a dense open subscheme $U\subset \mathbb{P}^1_F$ containing $0$, $1$ and a class $\gamma\in H^n(U,M)$ such that $\gamma(0)=\alpha$ and $\gamma(1)=\beta$. A class is $R$-trivial if it is $R$-equivalent to $0$ and we denote the set of such classes by $\RH^n(F,M)$.  
\end{definition}

$R$-equivalence is an equivalence relation on $H^n(F,M)$ and $\RH^n(F,M)$ is a subgroup. Moreover, $R$-trivial classes are preserved by morphisms of finite $p$-modules $M\to M'$, by the restriction $H^n(F,M)\to H^n(L,M)$ for a finite extension $L/F$, and by Shapiro's isomorphism $H^n(L,M)\simeq H^n(F,\mathrm{Ind}_{\Gamma_L}^{\Gamma_F}M)$.

\begin{lemma}\label{puresymbRtrivial}
Let $n$, $m$ be two positive integers and $M$ be a finite $p$-module such that $p^mM=0$. For any $a\in F^*$ and $\beta\in H^{n-1}(F,M)$, the cup product $(a)\cup\beta $ is $R$-trivial. In particular, pure symbols are $R$-trivial.  
\end{lemma}

\begin{proof}
Set $f=1+t(a-1)$ and let $U\subseteq\mathbb{A}^1_F$ be the open subscheme where $f$ is invertible (it contains $0$ and $1$). Let $(f)\in H^1(U,\Z/p^m\Z)$ be the Kummer class of $f\in\mathcal{O}(U)^*$, and set $\gamma=(f)\cup\pi^*\beta\in H^n(U,M)$, where $\pi\colon U\to\Spec F$ is the structure morphism. Since Kummer classes and cup products are compatible with pullbacks, for $x\in\{0,1\}$ we get $\gamma(x)=(f(x))\cup\beta$. Hence $\gamma(0)=(1)\cup\beta=0$ and $\gamma(1)=(a)\cup\beta$.
\end{proof}

\begin{remark}
Let $L/F$ be a finite extension and $n\ge1$. Under Shapiro's isomorphism
$H^n(F,\F_p[\Gamma_F/\Gamma_L])\simeq H^n(L,\F_p)$, sums of pure symbols over $L$ give $R$-trivial classes over $F$. By Kummer theory for $n=1$, and by the Bloch-Kato conjecture in general, these are all the classes: $H^n(F,P)$ is $R$-trivial for every permutation module $P$ with coefficients in $\F_p$.
\end{remark}

\begin{lemma}\label{bocklemma}
    Let $n\geq 1$ and $0\rightarrow A\rightarrow B\rightarrow C \rightarrow 0$ be a exact sequence of finite $p$-modules. If for every finite field extension $L/F$ the map $H^n(L,A)\rightarrow H^n(L,B)$ is injective, then the connecting homomorphism $\delta:H^n(F,C)\rightarrow H^{n+1}(F,A)$ vanishes on $\RH^n(F,C)$.
\end{lemma}

\begin{proof}
    Pick $\alpha\in \RH^n(F,C)$, so that there is a suited open $U\subset \mathbb{P}^1_F$ and $\gamma\in H^n(U,C)$ such that $\gamma(0)=0$ and $\gamma(1)=\alpha$. Denote by $\delta_U$ the connecting homomorphism of the exact sequence 
    
        \begin{equation}
        0\longrightarrow A \longrightarrow B \longrightarrow C \longrightarrow 0 \vspace{0.3cm}
    \end{equation}
    \noindent of sheaves over $U$ and set $\tilde{\gamma}=\delta_U(\gamma)\in H^{n+1}(U,A)$.

    \noindent As connecting homomorphisms commute with pullbacks, $\tilde{\gamma}(0)=\delta(\gamma((0))=0$ and $\tilde{\gamma}(1)=\delta(\alpha)$.

\noindent Let $Z=\PP^1_F\setminus U$, a finite set of closed points. By the Gysin sequence (see \cite[Ch. VI, §5]{Milne}, recall
that Tate twists are trivialized in our context), for every finite $p$-module $M$ there is an exact sequence
$$H^{n+1}(\PP^1_F,M)\longrightarrow H^{n+1}(U,M)
\xrightarrow{\ (\mathrm{Res}_s)_{s\in Z}\ }\bigoplus_{s\in Z}H^{n}(\kappa(s),M),$$
functorial in $M$. In particular, for every $s\in Z$ we have a commutative square
$$
\xymatrix{
H^{n+1}(U,A) \ar[r]^-{\mathrm{Res}_s} \ar[d] & H^{n}(\kappa(s),A) \ar[d] \\
H^{n+1}(U,B) \ar[r]^-{\mathrm{Res}_s} & H^{n}(\kappa(s),B).
}
$$
By the long exact sequence of cohomology on $U$, the image of $\tilde\gamma=\delta_U(\gamma)$
in $H^{n+1}(U,B)$ is zero. Hence the image of $\mathrm{Res}_s(\tilde\gamma)$ in
$H^n(\kappa(s),B)$ is zero, and $\mathrm{Res}_s(\tilde\gamma)=0$ by our assumption applied to
$L=\kappa(s)$. Therefore $\tilde\gamma=(\gamma_1)_{|U}$ for some
$\gamma_1\in H^{n+1}(\PP^1_F,A)$.

By the projective bundle formula,
$H^{n+1}(\PP^1_F,A)=\pi^*H^{n+1}(F,A)\oplus\xi\cup\pi^*H^{n-1}(F,A)$, where
$\pi\colon\PP^1_F\to\Spec F$ is the structure morphism and $\xi$ is the first Chern class of
$\mathcal{O}(1)$. For $x\in\PP^1(F)$ we have $x^*\pi^*=\mathrm{id}$ and $x^*\xi=c_1(x^*\mathcal{O}(1))=0$ since every line bundle on $\Spec F$ is
trivial. In particular $x^*\gamma_1$ does not depend on $x$ and
$$\delta(\alpha)=\tilde\gamma(1)=1^*\gamma_1=0^*\gamma_1=\tilde\gamma(0)=0.$$
\end{proof}

\section{Galois cohomology of iterated Laurent series}

From now on, $k$ is an algebraically closed field of characteristic $0$, $n$, $m$ are positive integers and $F=k((x_1))\cdots((x_{n+1}))$. We recall the following facts (see \cite{Serre} for (1) and
\cite{Wadsworth} for (2)--(4)).

\begin{enumerate}
\item The absolute Galois group $\Gamma_F$ is isomorphic to $\hat{\Z}^{n+1}$, and every open subgroup $U\subseteq\Gamma_F$ is also isomorphic to $\hat{\Z}^{n+1}$. The pro-$p$-Sylow subgroup $U_p$ of $U$ is isomorphic to $\Z_p^{n+1}$.
\item For every open subgroup $U\subseteq\Gamma_F$, the $\Z/p^m\Z$-module $$H^1(U,\Z/p^m\Z)=\mathrm{Hom}(U,\Z/p^m\Z)=\mathrm{Hom}(U_p,\Z/p^m\Z)$$ is free of rank $n+1$. For $U=\Gamma_F$, a basis is given by the Kummer classes $(x_1),\dots,(x_{n+1})$.
\item If $U\subseteq\Gamma_F$ is an open subgroup, the cup product induces an isomorphism of graded $\Z/p^m\Z$-algebras
$$\Lambda^\bullet_{\Z/p^m\Z}\,\mathrm{Hom}(U,\Z/p^m\Z)\overset{\sim}{\longrightarrow}H^\bullet(U,\Z/p^m\Z).$$
In particular, $H^{n+1}(U,\Z/p^m\Z)$ is free of rank one, $H^i(U,\Z/p^m\Z)=0$ for $i>n+1$, and
$F$ is of cohomological $p$-dimension $n+1$.
\item For open subgroups $U_1\subseteq U\subseteq\Gamma_F$, under the isomorphisms of (3),
the restriction $\mathrm{res}^U_{U_1}\colon H^i(U,\Z/p^m\Z)\to H^i(U_1,\Z/p^m\Z)$ is the $i$-th exterior
power of the natural restriction of characters $\mathrm{Hom}(U,\Z/p^m\Z)\to\mathrm{Hom}(U_1,\Z/p^m\Z)$.
\end{enumerate}

\begin{lemma}\label{lemmeker}
    Let $U\subset \Gamma_F$ be an open subgroup, and let $\chi\in H^1(U,\Z/p^m\Z)$ be a character whose reduction mod $p$ is not trivial. Set $U_{\chi}=\mathrm{ker}(\chi ~\mathrm{mod}~ p)$. For every positive integer $i$,
    $$\mathrm{ker}\left(\mathrm{res}_{U_{\chi}}^U\colon H^i(U,\Z/p^m\Z)\longrightarrow H^i(U_{\chi},\Z/p^m\Z)\right)=p^{m-1}\left(\chi \cup H^{i-1}(U,\Z/p^m\Z)\right).$$
\end{lemma}

\begin{proof}
As stated above, for any $i\geq 1$, $H^i(U,\Z/p^m\Z) = \Lambda^{\!i}\mathrm{Hom}(U_p,\Z/p^m\Z)$, where $U_p\simeq \Z_p^{n+1}$.
    Pick a lift $\tilde{\chi}\colon U_p\longrightarrow \Z_p$ of $\chi_{|U_p}$. Since $\chi$ is not trivial modulo $p$, there is an element $g_0\in U_p$ such that $\tilde{\chi}(g_0)$ is invertible. We may replace $g_0$ by $\tilde{\chi}(g_0)^{-1}g_0$ and thus assume $\tilde{\chi}(g_0)=1$, so that $U_p=\Z_pg_0\oplus \mathrm{ker}(\tilde{\chi})$.

    Now since $U_\chi\cap U_p=\{g\in U_p, \tilde\chi(g)\in p\Z_p\}$, we have $U_{\chi}\cap U_p=(p\Z_p)g_0\oplus \mathrm{ker}(\tilde{\chi})$: a character of $U_p$ (resp.\ $U_{\chi}\cap U_p$) is determined by its value on $g_0$ (resp.\ $pg_0$) and its restriction to $\mathrm{ker}(\tilde{\chi})$. Setting $\chi_2=\tilde{\chi}/p~~\mathrm{mod}~~p^m\colon U_p\cap U_{\chi}\longrightarrow \Z/p^m\Z$, we have
    $$\left\{\begin{array}{l}
         \mathrm{Hom}(U,\Z/p^m\Z)=\mathrm{Hom}(U_p,\Z/p^m\Z)=\Z/p^m\Z\cdot \chi \oplus \mathrm{Hom}(\mathrm{ker}(\tilde{\chi}),\Z/p^m\Z)  \\
         \mathrm{Hom}(U_{\chi},\Z/p^m\Z)=\mathrm{Hom}(U_p\cap U_{\chi},\Z/p^m\Z)=\Z/p^m\Z\cdot \chi_2 \oplus \mathrm{Hom}(\mathrm{ker}(\tilde{\chi}),\Z/p^m\Z)
    \end{array}\right.$$
and through this identification, the restriction map $\mathrm{res}_{U_{\chi}}^U:H^1(U,\Z/p^m\Z)\longrightarrow H^1(U_{\chi},\Z/p^m\Z)$ 
sends $\chi$ to $p\chi_2$ and its restriction to $\mathrm{Hom}(\mathrm{ker}(\tilde{\chi}),\Z/p^m\Z)$ is the identity.

Now every element of $\Lambda^i\mathrm{Hom}(U,\Z/p^m\Z)$ is uniquely written as $\chi\wedge \omega + \omega'$ with $\omega$ in $\Lambda^{i-1}\mathrm{Hom}(\mathrm{ker}(\tilde{\chi}),\Z/p^m\Z)$, $\omega'\in \Lambda^{i}\mathrm{Hom}(\mathrm{ker}(\tilde{\chi}),\Z/p^m\Z)$ so that 
$$\mathrm{res}_{U_{\chi}}^U(\chi\wedge \omega + \omega')=p\chi_2\wedge \omega +\omega'.$$
This vanishes if and only if $\omega'=0$ and $p\omega=0$, that is $\omega \in p^{m-1}\Lambda^{i-1}\mathrm{Hom}(\mathrm{ker}(\tilde{\chi}),\Z/p^m\Z)$. Consequently using $\chi\wedge\chi=0$ we have $$\mathrm{ker}\left(\mathrm{res}_{U_{\chi}}^U\right)=p^{m-1}\chi \wedge \left(\Lambda^{i-1}\mathrm{Hom}(\mathrm{ker}(\tilde{\chi}),\Z/p^m\Z)\right)=p^{m-1}\chi \wedge \left(\Lambda^{i-1}\mathrm{Hom}(U,\Z/p^m\Z)\right)$$
since $\Lambda^{i-1}\mathrm{Hom}(U,\Z/p^m\Z)=\chi\wedge\Lambda^{i-2}\mathrm{Hom}(\mathrm{ker}(\tilde{\chi}),\Z/p^m\Z)\oplus\Lambda^{i-1}\mathrm{Hom}(\mathrm{ker}(\tilde{\chi}),\Z/p^m\Z)$.
\end{proof}

\section{A quotient affine hyperplane module associated to $F$}

We keep the notation of the previous section, and set
$$L=F(x_1^{1/p},\dots,x_{n+1}^{1/p}),\qquad G=\mathrm{Gal}(L/F).$$
By the facts recalled in the previous section, the group $G$ is an $\F_p$-vector space of dimension $n+1$, so that we may consider the set $\mathcal{H}_G$ of linear hyperplanes of $G$, as well as the quotient affine hyperplane module $\mathcal{Q}_G$ of Definition~\ref{quotdef}.
Recall the exact sequence
\begin{equation}\label{seG}
0\longrightarrow\F_p\overset{\iota}{\longrightarrow}\bigoplus_{H\in\mathcal{H}_G}\F_p[G/H]
\longrightarrow\mathcal{Q}_G\longrightarrow0,
\end{equation}
where $\iota(1)=\bigl(\sum_{c\in G/H}c\bigr)_{H\in\mathcal{H}_G}$. It is a sequence of
$\F_p[G]$-modules, which we view as a sequence of $\Gamma_F$-modules through the quotient map $\pi\colon \Gamma_F\rightarrow G$. For every open subgroup $U\subseteq\Gamma_F$, we denote by $\delta_U\colon H^i(U,\mathcal{Q}_G)\longrightarrow H^{i+1}(U,\F_p)$ the connecting homomorphism in the long exact sequence of cohomology associated to \eqref{seG}.

\begin{proposition}\label{bock2}
    For any open subgroup $U\subset \Gamma_F$ and any $i\leq n$, the map $$\iota_*\colon H^i(U,\F_p)\rightarrow H^i\Bigl( U,\bigoplus_{H\in \mathcal{H}_G} \F_p[G/H]\Bigr)$$ is injective.
\end{proposition}

\begin{proof}For $i=0$ there is nothing to prove, so we assume $i\ge1$.

\textit{Case $\pi(U)\neq G$.} In this case $\pi(U)$ is contained in a hyperplane $H\in\mathcal{H}_G$, so that $U$ acts trivially on $G/H$. The $H$-component $\iota_H\colon\F_p\to\F_p[G/H]\simeq\F_p^{\,p}$ of $\iota$ is then the diagonal map of trivial $U$-modules, which is split. Hence $(\iota_H)_*$ is injective, and so is $\iota_*$.

\emph{Case $\pi(U)=G$.} Here
$\pi_{|U}^*\colon\mathrm{Hom}(G,\F_p)\to\mathrm{Hom}(U,\F_p)$ is injective, hence bijective, both spaces being of dimension $n+1$. So every nonzero character $\chi$ of $U$ is of the form $\psi\circ\pi_{|U}$ for a nonzero character $\psi$ of $G$, whose kernel is a hyperplane $H\in\mathcal{H}_G$, and any such $H$ arises in this way. For every such $H$ and $\chi$, the group $U$ acts transitively on $G/H$ with stabilizer $U_\chi=\ker\chi$, so that by Shapiro's lemma, $H^i(U,\F_p[G/H])=H^i(U_\chi,\F_p)$ and the $H$-component $(\iota_H)_*$ is the restriction $\mathrm{res}^U_{U_\chi}$. Therefore, by Lemma \ref{lemmeker} (for $m=1$),
$$\ker\iota_*=\bigcap_{H\in\mathcal{H}_G}\mathrm{ker}((\iota_H)_*)=\bigcap_{0\ne\chi\in H^1(U,\F_p)}\ker\bigl(\mathrm{res}^U_{U_\chi}\bigr)
=\bigcap_{0\ne\chi\in H^1(U,\F_p)}\chi\cup H^{i-1}(U,\F_p).$$
Choose a basis $e_1,\dots,e_{n+1}$ of $H^1(U,\F_p)$, so that the monomials
$e_J=\bigwedge_{j\in J}e_j$, $|J|=i$, form a basis of $H^i(U,\F_p)=\Lambda^iH^1(U,\F_p)$. For each $j$, the subspace $e_j\cup H^{i-1}(U,\F_p)$ is spanned by the $e_J$ with $j\in J$.
Hence an element of $\ker\iota_*$, which lies in all of these subspaces, only involves monomials $e_J$ with $J\supseteq\{1,\dots,n+1\}$. As $|J|=i<n+1$ there is no such monomial and therefore we get $\ker\iota_*=0$.
\end{proof}

\begin{proposition}\label{deltasurj}
    The map $\iota_*\colon H^{n+1}(F,\F_p)\longrightarrow H^{n+1}\Bigl(F,\bigoplus_{H\in \mathcal{H}_G} \F_p[G/H]\Bigr)$ vanishes.\\In particular, the connecting homomorphism $\delta_{\Gamma_F}\colon H^{n}(F,\mathcal{Q}_G)\to H^{n+1}(F,\F_p)$ is surjective.
\end{proposition}

\begin{proof}
Let $H\in\mathcal{H}_G$. As in the second case of the proof of Proposition~\ref{bock2} (with $U=\Gamma_F$), the $H$-component $(\iota_H)_*$ is the restriction $\mathrm{res}^{\Gamma_F}_{U_\chi}$, where $\chi$ is a nonzero character of $\Gamma_F$ with kernel $U_\chi=\pi^{-1}(H)$. Complete $\chi=e_1$ into a basis $e_1,\dots,e_{n+1}$ of $H^1(F,\F_p)$, so that $H^{n+1}(F,\F_p)=\F_p\cdot e_1\cup\dots\cup e_{n+1}$. Since restriction is compatible with cup products and $\mathrm{res}^{\Gamma_F}_{U_\chi}(\chi)=0$, we get $(\iota_H)_*(e_1\cup\dots\cup e_{n+1})=0$, hence $\iota_*=0$.
\end{proof}

\begin{proof}[Proof of Theorem \ref{main2}] We first show that $$\RH^{n}(F,\mathcal{Q}_G)=\mathrm{ker}\Bigl(\delta:H^{n}(F,\mathcal{Q}_G)\longrightarrow H^{n+1}(F,\F_p)\Bigr).$$
To show that classes in $\mathrm{ker}(\delta)$ are $R$-trivial, it is enough to prove that for all hyperplanes $H\subset G$, every class of $H^{n}(F,\F_p[G/H])$ is $R$-trivial by the long exact sequence of \eqref{seG}. Using Shapiro's lemma again, we are reduced to showing that all classes in $H^{n}(L^H,\F_p)$ are $R$-trivial. Since $H^{n}(L^H,\F_p)=\Lambda^{n}H^1(L^H,\F_p)$ is spanned by cup-products of $n$ classes of $H^1(L^H,\F_p)$, i.e.\ by pure symbols $(a_1)\cup \cdots \cup (a_{n})$ with $a_i\in (L^H)^*$, it follows from Lemma \ref{puresymbRtrivial}. 

As for the other inclusion, Proposition \ref{bock2} ensures for any finite field extension $K/F$ the injectivity of $\iota_*\colon H^{n}(K,\F_p)\longrightarrow H^{n}\Bigl(K,\bigoplus_{H\in \mathcal{H}_G} \F_p[G/H]\Bigr)$. It remains to apply Lemma \ref{bocklemma}.

To get the dimension statement, recall that $\delta$ is surjective by Proposition \ref{deltasurj} onto the line $H^{n+1}(F,\F_p)$, so that $\RH^{n}(F,\mathcal{Q}_G)$ is a hyperplane. Moreover, the exact sequence
$$0\longrightarrow H^{n}(F,\F_p)\overset{\iota_*}{\longrightarrow} H^{n}\Bigl(F,\bigoplus_{H\in \mathcal{H}_G} \F_p[G/H]\Bigr)\longrightarrow H^{n}(F,\mathcal{Q}_G)\overset{\delta}{\longrightarrow} H^{n+1}(F,\F_p)\longrightarrow 0$$
together with $\dim\Bigl(H^{n}(F,\bigoplus_{H\in \mathcal{H}_G} \F_p[G/H])\Bigr)=\sum_{H\in\mathcal{H}_G}\dim\Bigl(H^{n}(L^H,\F_p)\Bigr)=|\mathcal{H}_G|(n+1)$ and $\dim\Bigl(H^{n}(F,\F_p)\Bigr)=n+1$ indeed give $\dim\Bigl(H^{n}(F,\mathcal{Q}_G)\Bigr)=|\mathcal{H}_G|(n+1)-(n+1)+1$
\end{proof}

\section{$R$-trivial classes and corestrictions}

In this section we describe the group of $R$-trivial classes as sums of corestrictions of
cup products involving a Kummer class. From now on $F$ is any field of characteristic $0$ containing all roots of unity of $p$-power order, $n$ and $m$ are positive integers, and $M$ is a finite $p$-module with $p^mM=0$. Kummer classes are taken with coefficients in $\Z/p^m\Z$.

\subsection{Corestrictions}

Let $f\colon Y\to X$ be a finite \'etale morphism and let $M$ be a finite $p$-module, viewed
as a sheaf on $X$. Since $f_*$ is exact, we have $H^n(Y,M)=H^n(X,f_*f^*M)$, and the trace
map $f_*f^*M\to M$ induces the corestriction
$$\mathrm{cor}_{Y/X}\colon H^n(Y,M)\to H^n(X,M).$$
For $X=\Spec F$ and $Y=\Spec L$, with $L/F$ a finite extension, this is the usual
corestriction $\mathrm{cor}_{L/F}$ in Galois cohomology.

Since corestrictions commute with base change, they preserve $R$-trivial classes. Indeed, let $L/F$ be a finite
extension and $\beta\in H^n(U',M)$ with $U'\subseteq\mathbb{A}^1_L$ a dense open containing
$0$ and $1$. We may assume that $U'=U_L$ is the inverse image of a
dense open $U\subseteq\mathbb{A}^1_F$ containing $0$ and $1$, and then
$(\mathrm{cor}_{U_L/U}\,\beta)(x)=\mathrm{cor}_{L/F}(\beta(x))$ for $x\in\{0,1\}$.

\subsection{Residues}
Let $U\subseteq\mathbb{A}^1_F$ be a dense open subscheme and let $\beta\in H^n(U,M)$. We
denote by $\eta$ the generic point of the affine line under consideration, and by
$\beta_\eta$ the image of $\beta$ in $H^n(F(t),M)$. Let $s$ be a closed point of
$\mathbb{A}^1_F$, with residue field $\kappa(s)$ and associated discrete valuation $v_s$ on $F(t)$. Following \cite[\S7.13]{GMS}, the residue of $\beta_\eta$ at $v_s$ is a class
$$\mathrm{Res}_s(\beta)\in H^{n-1}(\kappa(s),M),$$
which vanishes if $s\in U$. Moreover, for $x\in U(F)$, the value $\beta(x)$ only depends on
$\beta_\eta$.\\We will use the following three facts (recall that $\pi$ is the structure morphism).

\begin{enumerate}
\item[(1)] If $U$ contains $0$ and $1$ and $\mathrm{Res}_s(\beta)=0$ for all closed point $s\in\mathbb{A}^1_F\setminus U$, then $\beta(0)=\beta(1)$. Indeed, $\beta_\eta$ then comes from a class $\beta_0\in H^n(F,M)$ by \cite[Theorem~9.3]{GMS}. Thus $\beta$ and the constant class $\pi^*\beta_0$ have the same image in $H^n(F(t),M)$, hence the same value at every $x\in U(F)$, namely $x^*\pi^*\beta_0=\beta_0$.

\item[(2)] \emph{Residues of symbols.} For $f\in\mathcal{O}(U)^*$ and
$\gamma\in H^{n-1}(F,M)$, we have
$$\mathrm{Res}_s\bigl((f)\cup\pi^*\gamma\bigr)=v_s(f)\,\gamma_{|\kappa(s)}.$$
Indeed, writing $f=uP^{v_s(f)}$ in the completion of $F(t)$ at $v_s$, with $P$ a uniformizer
and $u$ a unit, we get $(f)\cup\gamma=(u)\cup\gamma+(P)\cup v_s(f)\gamma$, where the first
term is unramified, and we conclude by \cite[Proposition~7.11]{GMS}.

\item[(3)] \emph{Residues of corestrictions.} Let $L/F$ be a finite extension, let
$U_L\subseteq\mathbb{A}^1_L$ be the inverse image of $U$, and consider the corestriction
$\mathrm{cor}_{U_L/U}$. For $\xi\in H^n(U_L,M)$ we have
$$\mathrm{Res}_s\bigl(\mathrm{cor}_{U_L/U}\,\xi\bigr)
=\mathrm{Res}_s\bigl(\mathrm{cor}_{L(t)/F(t)}\,\xi_\eta\bigr)
=\sum_{y\mapsto s}\mathrm{cor}_{\kappa(y)/\kappa(s)}\bigl(\mathrm{Res}_y\,\xi\bigr),$$
where $y$ runs over the closed points of $\mathbb{A}^1_L$ mapping to $s$. For the second equality, let $K_s$ be the completion of $F(t)$ at $v_s$. By \cite[Chap.~II, \S3, Th\'eor\`eme~1]{Serre}, the algebra $L(t)\otimes_{F(t)}K_s$ is the product of the completions $L(t)_y$ of $L(t)$ at the points $y$ above $s$, so that over $K_s$ the corestriction from $L(t)$ becomes the sum of the corestrictions from the fields $L(t)_y$ (double coset formula, \cite[(1.5.6)]{NSW}). By \cite[Proposition~8.6]{GMS}, residues commute with corestrictions for finite extensions of complete fields, so the residue of the $y$-th term is $\mathrm{cor}_{\kappa(y)/\kappa(s)}(\mathrm{Res}_y\,\xi)$.
\end{enumerate}

\subsection{$R$-trivial classes as sums of corestrictions of cup products}

\begin{theorem}\label{corsym}
We have
$$\RH^n(F,M)=\sum_{L/F}\mathrm{cor}_{L/F}\bigl((L^*)\cup H^{n-1}(L,M)\bigr),$$
where $L$ runs over the finite extensions of $F$ and $(L^*)$ denotes Kummer classes in $H^1(L,\Z/p^m\Z)$.
\end{theorem}

\begin{proof}
The right-hand side is contained in $\RH^n(F,M)$: a class $(a)\cup\gamma$ with $a\in L^*$ and
$\gamma\in H^{n-1}(L,M)$ is $R$-trivial over $L$ by Lemma~\ref{puresymbRtrivial}, and
corestrictions preserve $R$-trivial classes.

Conversely, let $\alpha\in\RH^n(F,M)$, and let $U$ and $\beta\in H^n(U,M)$ be such that
$\beta(0)=0$ and $\beta(1)=\alpha$. Removing the point at infinity if necessary, we may assume
$U\subseteq\mathbb{A}^1_F$. Let $Z=\mathbb{A}^1_F\setminus U$ and for $s\in Z$, set $\gamma_s=\mathrm{Res}_s(\beta)$. 

Let $s\in Z$ and set $L=\kappa(s)=F[t]/(P_s)$, where $P_s$ is the monic irreducible polynomial of $s$. Denote by $\theta_s\in L$ the class of $t$, viewed as a rational point of $\mathbb{A}^1_L$ lying over $s$. The function $t-\theta_s$ is invertible
on $U_L$, since $s\notin U$. Consider
$$\xi_s=(t-\theta_s)\cup\pi^*\gamma_s\in H^n(U_L,M).$$
By (2) applied over $L$, we have $\mathrm{Res}_{\theta_s}(\xi_s)=\gamma_s=\mathrm{Res}_s(\beta)$ and
$\mathrm{Res}_y(\xi_s)=0$ for every closed point $y\neq\theta_s$ of $\mathbb{A}^1_L$.
In particular setting
$$\beta_s=\mathrm{cor}_{U_L/U}(\xi_s)\in H^n(U,M),$$
we get from (3) that $\mathrm{Res}_s(\beta_s)=\gamma_s$ (the sum has only one nonzero term)
and $\mathrm{Res}_{s'}(\beta_s)=0$ for every $s'\in Z$ with $s'\neq s$. Evaluating $\beta_s$ we have, for $x\in \{0,1\}$, 
$$\beta_s(x)=\mathrm{cor}_{L/F}\bigl(\xi_s(x)\bigr)
=\mathrm{cor}_{L/F}\bigl((x-\theta_s)\cup\gamma_s\bigr),$$
so that
$$\beta_s(1)-\beta_s(0)=\mathrm{cor}_{L/F}\Bigl(\Bigl(\frac{\theta_s-1}{\theta_s}\Bigr)\cup
\gamma_s\Bigr).$$

We can now set $\tilde{\beta}=\beta-\sum_{s\in Z}\beta_s \in H^n(U,M)$. The residue of $\tilde{\beta}$ is trivial at every closed point of $Z$, hence $\tilde{\beta}(0)=\tilde{\beta}(1)$ by (1). This gives
$$\alpha=\beta(1)-\beta(0)=\sum_{s\in Z}\bigl(\beta_s(1)-\beta_s(0)\bigr)=\sum_{s\in Z} \mathrm{cor}_{\kappa(s)/F}\Bigl(\Bigl(\frac{\theta_s-1}{\theta_s}\Bigr)\cup
\gamma_s\Bigr).$$
\end{proof}

\begin{remark}
The final formula obtained in the proof is a variant of
\cite[Exercise~9.23]{GMS}.
\end{remark}

As a direct consequence, the next corollary extends Lemma~\ref{bocklemma}.

\begin{corollary}\label{corbock}
Let $0\to A\to B\to C\to0$ be an exact sequence of finite $p$-modules killed by $p^m$, with
connecting homomorphism $\delta\colon H^n(F,C)\to H^{n+1}(F,A)$. Then
$$\delta\bigl(\RH^n(F,C)\bigr)=\sum_{L/F}\mathrm{cor}_{L/F}\bigl((L^*)\cup K_L\bigr)\mbox{, where }K_L=\ker\bigl(H^n(L,A)\to H^n(L,B)\bigr),$$
where $L$ runs over finite extensions of $F$.
\end{corollary}

\begin{proof}
Observe that the connecting homomorphisms commute with corestrictions, and compute $\delta((x)\cup \gamma)=-(x)\cup\delta(\gamma)$ for $x\in L^*$ and $\gamma\in H^{n-1}(L,C)$.
Since $\delta\bigl(H^{n-1}(L,C)\bigr)=K_L$, the claim follows from
Proposition~\ref{corsym}.
\end{proof}


\begin{thebibliography}{99}
\bibitem{CTSan} J.-L. Colliot-Th\'el\`ene, J.-J. Sansuc, \emph{La $R$-\'equivalence sur les tores}, Ann. Sci. \'Ecole Norm. Sup. (4) 10, 175--229, 1977.
\bibitem{DCF2} C. De Clercq, M. Florence, \emph{Smooth profinite groups, I: geometrizing Kummer theory}, arXiv:2009.11130.
\bibitem{FR} M. Florence, Z. Reichstein, \emph{The rationality problem for forms of $\overline{M}_{0,n}$}, Bulletin of the London Math. Society, vol. 62 (1), 2018.
\bibitem{Milne} J.~S. Milne, \emph{\'Etale cohomology}, Princeton Math. Series 33, Princeton Univ. Press, 1980.
\bibitem{GMS} S. Garibaldi, A. Merkurjev, J.-P. Serre, \emph{Cohomological invariants in Galois cohomology}, University Lecture Series 28, Amer. Math. Soc., 2003.
\bibitem{NSW}Neukirch, J., Schmidt, A., Wingberg, K., \emph{Cohomology of number fields (2nd ed.)}, Springer, 2008.
\bibitem{Serre} J.-P. Serre, \emph{Local fields}, Graduate Texts in Math.~67, Springer, 1979.
\bibitem{Wadsworth} A. R. Wadsworth, \emph{$p$-Henselian fields: $K$-theory, Galois cohomology, and graded Witt rings}, Pacific J. Math. 105, no.~2, 473--496, 1983.
\end{thebibliography}
\end{document}